\documentclass[11pt,a4paper,reqno,twoside]{article}
\usepackage{hyperref}
\usepackage{amsfonts}
\usepackage{bbding}
\usepackage[all]{xy}
\usepackage{young}
\usepackage{multirow}
\usepackage{booktabs}
\usepackage{amsmath,amscd,amssymb,latexsym,srcltx,indentfirst,titlesec}
\usepackage{amsthm}

\usepackage{enumitem}

\setitemize[1]{itemsep=0pt,partopsep=0pt,parsep=\parskip,topsep=5pt}

\numberwithin{equation}{section}
\newtheorem{theorem}{Theorem}[section]
\newtheorem{lemma}[theorem]{Lemma}

\newtheorem{remark}[theorem]{Remark}

\newtheorem*{thma*}{Theorem A}
\newtheorem*{thmb*}{Theorem B}
\newtheorem*{thmc*}{Theorem C}
\newtheorem*{thmd*}{Theorem D}
\newtheorem*{thme*}{Theorem E}

\newtheorem{definition}[theorem]{Definition}
\newtheorem{question}[theorem]{Question}
\newtheorem*{thm*}{Main Theorem}
\allowdisplaybreaks

\begin{document}

\makeatletter

\newdimen\bibspace
\setlength\bibspace{2pt}   
\renewenvironment{thebibliography}[1]{%
 \section*{\refname 
       \@mkboth{\MakeUppercase\refname}{\MakeUppercase\refname}}%
     \list{\@biblabel{\@arabic\c@enumiv}}%
          {\settowidth\labelwidth{\@biblabel{#1}}%
           \leftmargin\labelwidth
           \advance\leftmargin\labelsep
           \itemsep\bibspace
           \parsep\z@skip     %
           \@openbib@code
           \usecounter{enumiv}%
           \let\p@enumiv\@empty
           \renewcommand\theenumiv{\@arabic\c@enumiv}}%
     \sloppy\clubpenalty4000\widowpenalty4000%
     \sfcode`\.\@m}
    {\def\@noitemerr
      {\@latex@warning{Empty `thebibliography' environment}}%
     \endlist}

\makeatother

\pdfbookmark[2]{ Some results on a question of M. Newman on isomorphic subgroups of solvable groups}{beg}

\title{{\bf Some results on a question of M. Newman on isomorphic subgroups of solvable groups} \footnotetext{\hspace*{-4 ex}
1. Department of Mathematics, Hubei University, Wuhan, $420062$, China\\
2. School of Mathematical Sciences, Peking University, Beijing, 100871, China\\
Supported by NSFC grants (11771129).\\
\hspace*{2 ex}Heguo Liu's E-mail: ghliu@hubu.edu.cn.\\
\hspace*{2 ex}Xingzhong Xu's E-mail: xuxingzhong407@hubu.edu.cn; xuxingzhong407@126.com.\\
\hspace*{2 ex}Jiping Zhang's E-mail: jzhang@pku.edu.cn.
}}

\author{{\small{Heguo Liu$^1$, Xingzhong Xu$^{1}$
, Jiping Zhang$^{2}$} } \\
}
\date{}
\maketitle

{\small

\noindent\textbf{Abstract.} {\small{In this paper, we focus on a question
of M. Newman on isomorphic subgroups of solvable groups. We get a reduction theorem
of this question: for each prime $q$, assume that this question holds for every characteristic $q$-groups,
then this question holds for every finite solvable groups. Using this reduction theorem,
we get some partial answers about this question.
 }}\\

\noindent\textbf{Keywords: }{\small {  solvable groups; maximal subgroups. }}

\noindent\textbf{Mathematics Subject Classification (2020):}  \ 20D10.}

\section{\bf Introduction}

Recently, G. Glauberman, I.M. Isaacs and G.R. Robinson's works \cite{IR, GR} focus on a question which posted
by Moshe Newman who asked the following:

\begin{question}\cite{IR, GR} Whether can it  ever happen that a finite solvable group $G$ has isomorphic subgroup $H$ and $K$,
where $H$ is maximal and $K$ is not?
\end{question}

In 2015, I.M. Isaacs and G.R. Robinson have done some partial results as follows.
\begin{theorem}\cite[Theorem A, Theorem B]{IR} Let $H$ be a maximal subgroup of a solvable group $G$, and suppose that $K\leq G$
and $K\cong H$. If $H$ has a Sylow tower, or a Sylow $2$-subgroup of $H$ is abelian, then $K$ is maximal in $G$.
\end{theorem}

And recently, G. Glauberman and G.R. Robinson get some partial results about the structure of $G$
 when there exists a negative answer of Question 1.1.

\begin{theorem}\cite[Theorem A]{GR} Let $H$ be a maximal subgroup of the finite solvable group $G$ and suppose that $|G : H|=p^a$ where $p$ is a prime and $a$ is a positive integer. Let $K$ be a subgroup of $G$ which is isomorphic to $H$.
Suppose that $K$ is not maximal in $G$. Then $p \leq 3$, and, for $q =5-p$, we have
$$O_{q'}(H)=O_{q'}(G)=O_{q'}(K)$$ and, for $G^{\ast}= G/O_{q'}(G)$, etc., $H^{\ast}$ and $K^{\ast}$
 are isomorphic subgroups of $G^{\ast}$ with $H^{\ast}$ maximal and $K^{\ast}$ not maximal.
\end{theorem}
The above result use the remarkable theorem of G. Glauberman(see \cite{G, S}).
And it tells us that Question 1.1 is ture when $p\geq 5$ where $|G: H|=p^a$ for some positive integer $a$.
So we will only need to discuss this question in cases that $p\leq 3$.

Depending on some results of the above authors' works, we find that a class of finite groups is important for
Question 1.1.
This class of finite groups is of $characteristic~ l$.
Here, $l$ is a prime number. Recall that a finite group $G$ is said to be of $characteristic~ l$ if $C_G(O_l(G)) \leq O_l(G).$
We have a reduction theorem for Question 1.1 as follows.

\begin{thma*}For each prime $q$, assume that Question 1.1 holds for every characteristic $q$-groups $G$. Then Question 1.1 holds for every
finite solvable groups.
\end{thma*}

If $G$ is of $characteristic~ q$, then we have $C_G(O_q(G)) \leq O_q(G)$. So, we can find that
$$\mathrm{Aut}_G(O_q(G))=N_G(O_q(G))/C_G(O_q(G))=G/Z(O_q(G)).$$
We can find some information of $G$ from the $\mathrm{Aut}(O_q(G))$. Especially, it becomes useful when
$O_q(G)$ is small or abelian.

\begin{theorem}  Let $G$ be a finite solvable group and $G$ has isomorphic subgroup $H$ and $K$.
Let $H$ is maximal subgroup of $G$, we can set $|G: H|=p^n$. Let $p\leq 3$ and $q=5-p$. Let
$Q\in Syl_q(H)$. If $|G|_q\leq q^4$, then $K$ is also maximal.
\end{theorem}

Recall that if $G$ is a $p$-soluble group, the $p$-length $l_p(G)$ is
the number of factors
of the lower $p$-series of $G$ that are $p$-groups(see \cite[p.227]{Go}).

\begin{theorem} Let $G$ be a finite solvable group and $G$ has isomorphic subgroup $H$ and $K$.
Let $H$ is maximal subgroup of $G$. If $l_p(G)\leq 1$, then $K$ is also maximal.
\end{theorem}

 In the other opinion, a model of a constrained fusion systems is also of  $characteristic~ q$ for some
 prime number $q$. By Theorem A, we can get the following theorem.
 \begin{thmb*} Let $G$ be a finite solvable group and $G$ has isomorphic subgroup $H$ and $K$.
 Let $H$ is maximal subgroup of $G$, we can set $|G: H|=p^n$. Let $p\leq 3$ and $q=5-p$. Let
 $Q\in Syl_q(H)$. If $\mathcal{F}_Q(H)\unlhd \mathcal{F}_Q(G)$, then $K$ is also maximal.
 \end{thmb*}

$Structure~ of ~ the~ paper:$ After recalling preliminary results, we give proofs of Theorem A, Theorem 1.4 and 1.5
in Section 2. And in Section 3, we give a proof of Theorem B.

\section{\bf Preliminary results, and proofs of Theorem A, Theorem 1.4 and 1.5}

The following lemmas are very useful to get the proof of Theorem A.

\begin{lemma}\cite[Lemma 2]{IR} Let $G$ be a solvable group and $H\leq G$,  where $|G: H|$ is power of
a prime $p$. Then $O_{p}(G)\cap H= O_{p}(H)$.
\end{lemma}

\begin{lemma}\cite[Theorem 3]{IR} Let $H$ be a maximal subgroup of a solvable group $G$ with index a power of the prime $p$,
and suppose that $K\leq G$ and $K\cong H$. If $O_p(G)\nleq H$, then $K$ is maximal in $G$.
\end{lemma}

\begin{theorem}\cite[Theorem A]{GR} Let $H$ be a maximal subgroup of the finite solvable group $G$ and suppose that $|G : H|=p^a$ where $p$ is a prime and $a$ is a positive integer. Let $K$ be a subgroup of $G$ which is isomorphic to $H$.
Suppose that $K$ is not maximal in $G$. Then $p \leq 3$, and, for $q =5-p$, we have
$$O_{q'}(H)=O_{q'}(G)=O_{q'}(K)$$ and, for $G^{\ast}= G/O_{q'}(G)$, etc., $H^{\ast}$ and $K^{\ast}$
 are isomorphic subgroups of $G^{\ast}$ with $H^{\ast}$ maximal and $K^{\ast}$ not maximal.
\end{theorem}

\begin{theorem}\cite[Theorem B]{GR} Let $H$ be a maximal subgroup of the finite solvable group $G$ and suppose that
$|G : H|=p^a$ where $p\leq 3$ is a prime and $a$ is a positive integer. Let $K$ be a subgroup of $G$ which is isomorphic to $H$.
 Suppose that $K$ is not maximal in $G$ and that $F(H)$, $F(K)$ and $F(G)$ are all $q$-groups, where $q =5-p$.
 Let $Q$ be a Sylow $q$-subgroup of $H$. Then $G$ has a homomorphic image $G^{\ast}$ such that $H^{\ast}$ and $K^{\ast}$
  (the respective images of $H$ and $K$) are isomorphic subgroups of $G^{\ast}$ with $H^{\ast}$ maximal and $K^{\ast}$ not maximal, and with $F(G^{\ast}), F(H^{\ast})$ and $F(K^{\ast})$ all $q$-groups. Furthermore, $O_{\{2,3\}}(K^{\ast})$ involves $Qd(q)$ and no non-identity characteristic subgroup of $Q^{\ast}$ is normal in $H^{\ast}$.
\end{theorem}

\begin{remark}By above two theorems, we can find that Question 1.1 holds when $p\geq 5$.
So we will only need to consider the cases when $p\leq 3$. Here, $p$ is a prime satisfied that
$|G: H|=p^a$ where $a$ is a positive integer.
\end{remark}

Now, we will prove Theorem A as follows. This can be seem as a corollary of \cite[Theorem A]{GR} and \cite[Theorem B]{GR}.

\begin{thma*}For each prime $q$, assume that Question 1.1 holds for every characteristic $q$-groups $G$. Then Question 1.1 holds for every
finite solvable groups.
\end{thma*}

\begin{proof} Suppose that $(G, H, K)$ is a counterexample. Since $H$ is maximal in a solvable group
$G$, we can set $|G:H|=p^n$ for some prime $p$ and positive integer $n$.

{\bf Case 1.} $O_p(G)\neq 1$. By \cite[Theorem 3]{IR}, we have $O_{p}(G)\leq H$.
By \cite[Lemma 2]{IR}, we have
$$O_{p}(G)=O_{p}(G)\cap H= O_{p}(H),~~~~O_{p}(G)\cap K= O_{p}(K).$$
Since $H\cong K$, we have $O_{p}(H)\cong O_{p}(K)$. Hence, $O_{p}(G)\leq K$.
Now, we focus on $(G/O_{p}(G), H/O_{p}(G), K/O_{p}(G))$, we can see that
$K/O_{p}(G)$ is maximal in $G/O_{p}(G)$ because $(G, H, K)$ is a counterexample. So $K$ is maximal in $G$.
That is a contradiction.

{\bf Case 2.} $O_p(G)=1$. First, since $G$ is solvable, we have
 $O_{p'}(G)\neq 1$.

By \cite[Theorem A]{GR}, we can see that $O_{q'}(G)=1$ because $(G, H, K)$ is a counterexample.
So the Fitting subgroup $F(G)=O_{q}(G)$ and $O_{q}(G)\neq 1$ because $O_{p'}(G)\neq 1$. Since
$C_{G}(F(G))\leq F(G)$, we have$C_{G}(O_{q}(G))\leq O_{q}(G)$.
It implies $G$ is of characteristic $q$-group. But by the assumption,
we know that Question 1.1 holds for every characteristic $q$-groups $G$.
Hence, that is a contradiction.

So, we complete the proof.
\end{proof}

Now, we will prove Theorem 1.4 as follows.

\begin{theorem}  Let $G$ be a finite solvable group and $G$ has isomorphic subgroup $H$ and $K$.
Let $H$ is maximal subgroup of $G$, we can set $|G: H|=p^n$. Let $p\leq 3$ and $q=5-p$. Let
$Q\in Syl_q(H)$. If $|G|_q\leq q^4$, then $K$ is also maximal.
\end{theorem}

\begin{proof} Suppose that $(G, H, K)$ is a counterexample. Since $H$ is maximal in a solvable group
$G$, we can set $|G:H|=p^n$ for some prime $p$ and positive integer $n$.

{\bf Case 1.} $O_p(G)\neq 1$. By \cite[Theorem 3]{IR}, we have $O_{p}(G)\leq H$.
By \cite[Lemma 2]{IR}, we have
$$O_{p}(G)=O_{p}(G)\cap H= O_{p}(H),~~~~O_{p}(G)\cap K= O_{p}(K).$$
Since $H\cong K$, we have $O_{p}(H)\cong O_{p}(K)$. Hence, $O_{p}(G)\leq K$.

Now, we focus on $(G/O_{p}(G), H/O_{p}(G), K/O_{p}(G))$.
Since $H/O_{p}(G)\cong K/O_{p}(G)$ and $|G/O_{p}(G)|_q=|G|_q\leq q^4$,  we can see that
$K/O_{p}(G)$ is maximal in $G/O_{p}(G)$ because $(G, H, K)$ is a counterexample. So $K$ is maximal in $G$.
That is a contradiction.

{\bf Case 2.} $O_p(G)=1$. First, since $G$ is solvable, we have
 $O_{p'}(G)\neq 1$.

By \cite[Theorem A]{GR}, we can see that $O_{q'}(G)=1$ because $(G, H, K)$ is a counterexample.
So $F(G)=O_{q}(G)$ and $O_{q}(G)\neq 1$ because $O_{p'}(G)\neq 1$. Here, we have $C_{G}(O_{q}(G))\leq O_{q}(G)$.
Since $|G:H|=p^n$, we have $O_{q}(G)\leq H$. Similarly, $O_{q}(G)\leq K$.
By the assumption $|G|_q\leq q^4$, we can discuss as follows.

{\bf Case 2.1.} $|O_{q}(G)|=q^4$. Since $H\cong K$, we can set an isomorphic map $\alpha: K\to H$.
So $\alpha$ set $O_{q}(G)$ to $\alpha(O_{q}(G))$. Here, $\alpha(O_{q}(G))\leq H, O_{q}(G)\leq H$.
Hence $\alpha(O_{q}(G))O_{q}(G)=O_{q}(G)$ because $|G|_q\leq q^4$. So $\alpha(O_{q}(G))=O_{q}(G)$.
Then we can consider $(G/O_{q}(G), H/O_{q}(G), K/O_{q}(G))$.
Since $H/O_{q}(G)= H/\alpha(O_{q}(G)\cong K/O_{q}(G)$, we have $K/O_{q}(G)$ is maximal in $G/O_{q}(G)$
because $(G, H, K)$ is a counterexample. So $K$ is maximal in $G$. That is a contradiction.

{\bf Case 2.2.}  $|O_{q}(G)|=q^3$. Since $H\cong K$, we can set an isomorphic map $\alpha: K\to H$.
So $O_{q}(G)$ is sent to $\alpha(O_{q}(G))$ by map $\alpha$.
If $\alpha(O_{q}(G))=O_{q}(G)$, then we can consider $(G/O_{q}(G), H/O_{q}(G), K/O_{q}(G))$.
Since $H/O_{q}(G)= H/\alpha(O_{q}(G))\cong K/O_{q}(G)$, we have $K/O_{q}(G)$ is maximal in $G/O_{q}(G)$
because $(G, H, K)$ is a counterexample. So $K$ is maximal in $G$. That is a contradiction.
Hence,  $\alpha(O_{q}(G))\neq O_{q}(G)$, we have
$$\alpha(O_{q}(G))O_{q}(G)\gneq O_q(G).$$
Since $|O_{q}(G)|=q^3$ and $|G|_q\leq q^4$, we have
$\alpha(O_{q}(G))O_{q}(G)\in \mathrm{Syl}_q(G)$.
Set $Q:=\alpha(O_{q}(G))O_{q}(G)$, we have $\alpha^{-1}(Q)\in\mathrm{Syl}_q(G)$. There exists
$g\in G$ such that $Q=\alpha^{-1}(Q)^g$. Now we can consider $(G, H, K^g)$.
We have $Q=\alpha^{-1}(Q)^g\leq K^g$. So $Q$ is sent to $Q$ by morphism 
$$\xymatrix@C=0.5cm{
  K^g \ar[rr]^{c_{g^{-1}}} && K \ar[rr]^{\alpha} && H}.$$
Since $Q\unlhd H$, we have $Q\unlhd K^g$. If $K^g\leq H$, we can see that $K^g$ is maximal in $G$. That is a contradiction.
Hence,  $K^g\nleq H$. So $Q\unlhd G$. Now,
we can consider $(G/Q, H/Q, K^g/Q)$. Since
$$K^g/Q\cong K/\alpha^{-1}(Q)\cong H/Q,$$
we have $K^g$ is maximal in $G$. That is contradiction.

{\bf Case 2.3.}  $|O_{q}(G)|\leq q^2$. By similar reason of the above case, we
can set $\alpha(O_{q}(G))\neq O_{q}(G)$ and $1\neq\alpha(O_{q}(G))\cap O_{q}(G)\lneq O_{q}(G)$.
 Set $N_1=\alpha(O_{q}(G))\cap O_{q}(G)$ and
$N_2=\alpha(N_1)\cap O_{q}(G)$. It is easy to see that
$N_2\leq N_1$. Since $|O_{q}(G)|\leq q^2$,
we have either $N_2=1$ or $N_2=N_1$.

If  $N_2=N_1$, we have $N_1=N_2=\alpha(N_1)\cap O_{q}(G)$. So $N_1=\alpha(N_1)$.
Since $N_1\unlhd H$, we have $N_1\unlhd K$. So, $N_1\unlhd G$.
Now, we consider $(G/N_1, H/N_1, K/N_1)$, we have $K$ is maximal in $G$. That is contradiction.

If $N_2=1$, we have
$\alpha(N_1)\cap O_{q}(G)=1$. But $G/O_{q}(G)$ is isomorpical to a subgroup of $\mathrm{Aut}(O_{q}(G))$, we have $|G|_q\leq q^3$.
Hence $\alpha(O_{q}(G))O_{q}(G)\in \mathrm{Syl}_q(G)$. So, by the similar reason of above case, we can get
a contradiction.

So, we complete the proof.
\end{proof}

Now, we will prove Theorem 1.5 as follows.
First, recall that if $G$ is a $p$-soluble group, the $p$-length $l_p(G)$ is
the number of factors
of the lower $p$-series of $G$ that are $p$-groups(see \cite[p.227]{Go}).
 \begin{theorem} Let $G$ be a finite solvable group and $G$ has isomorphic subgroup $H$ and $K$.
Let $H$ is maximal subgroup of $G$. If $l_p(G)\leq 1$, then $K$ is also maximal.
\end{theorem}

\begin{proof} Suppose that $(G, H, K)$ is a counterexample. Since $H$ is maximal in a solvable group
$G$, we can set $|G:H|=p^n$ for some prime $p$ and positive integer $n$.

{\bf Case 1.} $O_p(G)\neq 1$. By \cite[Theorem 3]{IR}, we have $O_{p}(G)\leq H$.
By \cite[Lemma 2]{IR}, we have
$$O_{p}(G)=O_{p}(G)\cap H= O_{p}(H),~~~~O_{p}(G)\cap K= O_{p}(K).$$
Since $H\cong K$, we have $O_{p}(H)\cong O_{p}(K)$. Hence, $O_{p}(G)\leq K$.
Now, we focus on $(G/O_{p}(G), H/O_{p}(G), K/O_{p}(G))$. 

Since $l_p(G)\leq 1$, we have $SO_{p'}(G)\unlhd G$ for some Sylow $p$-subgroup of $G$.
We can see that $O_{p'}(G)\leq O_{p, p'}(G)$, so
$$SO_{p, p'}(G)=SO_{p'}(G)O_{p, p'}(G)\unlhd G.$$
Hence $l_{p}(G/O_{p}(G))\leq 1$. So, we can see that
$K/O_{p}(G)$ is maximal in $G/O_{p}(G)$ because $(G, H, K)$ is a counterexample. Hence, $K$ is maximal in $G$.
That is a contradiction.

{\bf Case 2.} $O_p(G)=1$. Since $G$ is solvable, we have
$F(G)\leq O_{p'}(G)\neq 1$. And $C_{G}(O_{p'}(G))\leq O_{p'}(G)$.

Now, we assert that $O_{p'}(G)\leq H$. If $O_{p'}(G)\nleq H$, thus $O_{p'}(G)\cap H\lneq O_{p'}(G)$.
By $$\frac{|HO_{p'}(G)|}{|H|}=\frac{|O_{p'}(G)|}{|O_{p'}(G)\cap H|},$$
we have $r||HO_{p'}(G): H|$ for some prime $r$ which is not $p$. That is a contradiction to $|G: H|=p^n$.

Hence, $O_{p'}(G)\leq H$. Similarly, we have $O_{p'}(G)\leq K$ because $|G:K|=|G:H|=p^n$.

First, we assert that $O_{p'}(G)$ is not a Hall $p'$-subgroup of $G$.
Else, $H/O_{p'}(G)\cong K/O_{p'}(G)$. Then we can get a contradiction by induction.

Since $l_p(G)\leq 1$, for each $S\in \mathrm{Syl}_p(G)$, we have $T:=SO_{p'}(G)=O_{p',p}(G)\unlhd G$.
And $|G: H|=p^n$, we have $T\nleq H$. Similarly, $T\nleq K.$
Now, we can see that
$$H\cap T=H\cap SO_{p'}(G)=(H\cap S)O_{p'}(G)\unlhd H.$$
Since $S\nleq H$, thus $N_{S}(H\cap S)\gneq H\cap S$.
So let $x\in N_{S}(H\cap S)- H\cap S$, then
$$((H\cap S)O_{p'}(G))^x=(H\cap S)^x O_{p'}(G)=(H\cap S)O_{p'}(G).$$
But $x\notin H$ and $H$ is maximal in $G$. Hence,
we have $$(H\cap S)O_{p'}(G)\unlhd G$$ because $(H\cap S)O_{p'}(G)\unlhd H.$

Let $R\in \mathrm{Syl}_p(H)$, there exists $t\in G$ such that
$R\leq S^t$. For $S^t$, we have $S^tO_{p'}(G)=SO_{p'}(G)\unlhd G$. Then
$$(H\cap S^t)O_{p'}(G)\unlhd G$$
and $H\cap S^t\geq R$. So $H\cap S^t=R\in  \mathrm{Syl}_p(H)$.

Now, we replace $S^t$ by $S$. That means
$$(H\cap S)O_{p'}(G)\unlhd G~
\mathrm{and} ~H\cap S \in  \mathrm{Syl}_p(H).$$

{\bf Case 2.1.} $(K\cap S)O_{p'}(G)\leq H.$ Then $(K\cap S)O_{p'}(G)\leq (H\cap S)O_{p'}(G).$
We know that $G=KSO_{p'}(G)=HSO_{p'}(G)$ and $SO_{p'}(G)\unlhd G$. So
$$K/((K\cap S)O_{p'}(G))\cong G/SO_{p'}(G)\cong H/((H\cap S)O_{p'}(G)).$$
Since $K\cong H$, we have $|(K\cap S)O_{p'}(G)|=|(H\cap S)O_{p'}(G)|$. Then
$$(K\cap S)O_{p'}(G)=(H\cap S)O_{p'}(G).$$
Now, for $(G/((H\cap S)O_{p'}(G)), K/((H\cap S)O_{p'}(G)), H/((H\cap S)O_{p'}(G)))$, we assert that
 $l_p(G/((H\cap S)O_{p'}(G)))\leq 1$. 
 Since $$O_{p'}(G/((H\cap S)O_{p'}(G)))\cdot SO_{p'}(G)/((H\cap S)O_{p'}(G))\unlhd G/((H\cap S)O_{p'}(G)),$$
we have $l_p(G/((H\cap S)O_{p'}(G)))\leq 1$.
So $K/((H\cap S)O_{p'}(G))$ is maximal in $G/((H\cap S)O_{p'}(G))$. Hence,
$K$ is maximal in $G$. That is a contradiction.

{\bf Case 2.2.} $(K^u\cap S)O_{p'}(G)\nleq H$ for each $u\in G$.
Since $H$ is maximal in $G$, we have $((K^u\cap S)O_{p'}(G))H=G$.
We assert that $K^u(H\cap S)=G$.
Since
$$|G|=|H(K^u\cap S)|=\frac{|H||K^u\cap S|}{|K^u\cap H\cap S|}$$
for each $u\in G$, we can choose $u_0$ such that $K^{u_0}\cap S\in \mathrm{Syl}_p(K^{u_0}).$
So $$\frac{|H||K^{u_0}\cap S|}{|K^{u_0}\cap H\cap S|}=\frac{|K^{u_0}||H\cap S|}{|K^{u_0}\cap H\cap S|}=|K^{u_0}(H\cap S)|$$
because $K^{u_0}\cong H$. Hence, $K^{u_0}(H\cap S)=G$.

Now, we replace $K^{u_0}$ by $K$. That means $K(H\cap S)=G.$
Set $V=(H\cap S)O_{p'}(G)$ which is a normal subgroup of $G$.
Set $Y:=H\cap K$ and $\alpha(Y)=X$ where $\alpha: K\to H$ is an isomorphic map.
First, we assert that $Y$ is maximal in $K$. Since $KV=G$, there exists an isomorphism $\phi: G/V\to K/K\cap V$.
And we can see that $\phi(H/V)=(H\cap K)/(K\cap V)= Y/(K\cap V).$ Since $H/V$ is maximal in $G/V$,
we have $Y/(K\cap V)$ is maximal in $K/(K\cap V).$ Hence, $Y$ is maximal in $K$, as wanted.

Then $X$ is maximal in $H$. Since $H\cap S\in \mathrm{Syl}_p(H)$ and $(H\cap S)O_{p'}(G)\unlhd H$, we have
$l_p(H)\leq 1$. By induction we have $Y$ is also maximal in $H$.

 Let $K\leq L\lneq G$. Then $H\geq L\cap H\geq H\cap K=Y$. If $L\cap H=H$, then $H\leq L$. So
 $L=G$. Hence, $L\cap H= H\cap K$. And
 $$L=L\cap G= L\cap KV=K(L\cap V)=K(L\cap ((H\cap S)O_{p'}(G))).$$
 But $L\cap ((H\cap S)O_{p'}(G))=(L\cap (H\cap S))O_{p'}(G)=(K\cap H\cap S)O_{p'}(G)\leq K$.
 Hence, $L\leq K$. That means $K$ is maximal in $G$. That is a contradiction.

So, we complete the proof.
\end{proof}

\section{\bf Notation of fusion systems, and proof of Theorem B}
In this section we collect some known results that will be needed later.  For the background theory of
fusion systems, we refer to \cite{AsKO, BLO1, BLO2}.

\begin{definition} A $fusion ~ system$ $\mathcal{F}$ over a finite $p$-group
$S$ is a category whose objects are the subgroups of $S$, and whose
morphism sets $\mathrm{Hom}_{\mathcal{F}}(P,Q)$ satisfy the
following two conditions:

\vskip 0.3cm

(a) $\mathrm{Hom}_{S}(P,Q)\subseteq
\mathrm{Hom}_{\mathcal{F}}(P,Q)\subseteq\mathrm{Inj}(P,Q)$ for all
$P,Q\leq S$.

\vskip 0.3cm

(b) Every morphism in $\mathcal{F}$ factors as an isomorphism in
$\mathcal{F}$ followed by an inclusion.
\end{definition}

\begin{definition} Let $\mathcal{F}$ be a fusion system over a $p$-group
$S$.

$\bullet$ Two subgroups $P,Q$ are $\mathcal{F}$-$conjugate$ if they
are isomorphic as objects of the category $\mathcal{F}$. Let
$P^{\mathcal{F}}$ denote the set of all subgroups of $S$ which are
$\mathcal{F}$-conjugate to $P$.
Since $\mathrm{Hom}_{\mathcal{F}}(P,P)\subseteq\mathrm{Inj}(P,P)$,
we usually write $\mathrm{Hom}_{\mathcal{F}}(P,P)=\mathrm{Aut}_{\mathcal{F}}(P)$ and
$\mathrm{Hom}_{S}(P,P)=\mathrm{Aut}_{S}(P)$.

$\bullet$ A subgroup $P\leq S$ is $fully~automised$ in $\mathcal{F}$
if $\mathrm{Aut}_{S}(P)\in
\mathrm{Syl}_{p}(\mathrm{Aut}_{\mathcal{F}}(P))$.

$\bullet$ A subgroup $P\leq S$ is $receptive$ in $\mathcal{F}$ if it
has the following property: for each $Q\leq S$ and each $\varphi\in
\mathrm{Iso}_{\mathcal{F}}(Q, P)$, if we set
$$N_{\varphi}=\{g\in N_{S}(Q)|\varphi \circ c_{g}\circ \varphi^{-1}\in \mathrm{Aut}_{S}(P)\},$$
then there is $\overline{\varphi}\in
\mathrm{Hom}_{\mathcal{F}}(N_{\varphi},S)$ such that
$\overline{\varphi}|_{Q}=\varphi$. (where $c_{g}:x\longmapsto g^{-1}xg$ for $g\in S$)

$\bullet$ A fusion system $\mathcal{F}$ over a $p$-group $S$ is
$saturated$ if each subgroup of $S$ is $\mathcal{F}$-conjugate to a
subgroup which is fully automised and receptive.

\end{definition}

\begin{definition} Let $\mathcal{F}$ be a fusion system over a $p$-group
$S$.

$\bullet$ A subgroup $P\leq S$ is $fully~normalized$ in
$\mathcal{F}$ if $|N_{S}(P)|\geq |N_{S}(Q)|$ for all $Q\in
P^{\mathcal{F}}$.

$\bullet$ A subgroup $P\leq S$ is $\mathcal{F}$-$centric$ if
$C_{S}(Q)=Z(Q)$ for $Q\in P^{\mathcal{F}}$.

$\bullet$ Let $\mathcal{F}^{c}$ denote the full subcategory of $\mathcal{F}$ whose objects are
$\mathcal{F}$-centric,

$\bullet$ Let $\mathcal{F}^{f}$ denote the full subcategory of $\mathcal{F}$ whose objects are
 fully~normalized in $\mathcal{F}$.

$\bullet$ A subgroup $P\leq S$ is $normal$ in $\mathcal{F}$ (denoted
$P\trianglelefteq \mathcal{F}$) if for all $Q,R\in S$ and all
$\varphi\in \mathrm{Hom }_{\mathcal{F}}(Q,R)$, $\varphi$ extends to
a morphism $\overline{\varphi}\in \mathrm{Hom
}_{\mathcal{F}}(QP,RP)$ such that $\overline{\varphi}(P)=P$.
Moreover, $O_{p}(\mathcal{F})$ denotes the largest subgroup of $S$
which is normal in $\mathcal{F}$.
\end{definition}

\begin{definition}\cite[I, Definition 6.1]{AsKO} Let $\mathcal{F}$ a saturated fusion system over a finite
$p$-group $S$. Let $\mathcal{E}$ be a subsystem of $\mathcal{F}$ over a subgroup $T$ of $S$.

$\bullet$ Define $\mathcal{E}$ to be $\mathcal{F}$-invariant if:

(I1) $T$ is strongly closed in $S$ with respect to $\mathcal{F}$;

(I2) For each $P\leq Q\leq T$, $\phi\in \mathrm{Hom}_{\mathcal{E}}(P, Q)$, and
$\alpha\in \mathrm{Hom}_{\mathcal{F}}(Q, S)$, $\phi^{\alpha}\in \mathrm{Hom}_{\mathcal{E}}(\alpha(P), T)$.

If $\mathcal{E}$ is saturated, we call that

$\bullet$ A subsystem $\mathcal{E}\subseteq \mathcal{F}$ is weakly normal in $\mathcal{F}$
($\mathcal{E}\dot{\unlhd} \mathcal{F}$) if $\mathcal{E}$ is saturated and $\mathcal{E}$
is $\mathcal{F}$-invariant.

$\bullet$ A weakly normal subsystem $\mathcal{E}\dot{\unlhd} \mathcal{F}$ is  normal in $\mathcal{F}$ if:

(N1) Each $\phi\in \mathrm{Aut}_{\mathcal{E}}(T)$ extends to $\hat{\phi}\in \mathrm{Aut}_{\mathcal{F}}(TC_S(T))$ such
that $[\hat{\phi}, C_S(T)]\leq Z(T)$.

We write $\mathcal{E}\unlhd \mathcal{F}$ to indicate that $\mathcal{E}$ is normal in $\mathcal{F}$.

$\bullet$ $\mathcal{F}$ is simple if it contains no proper nontrivial normal fusion subsystem.

$\bullet$ Define $O^p(\mathcal{F})$ to be the minimal normal subsystem of $\mathcal{F}$ which has $p$-power index in $\mathcal{F}$
(See \cite[I, Theorem 7.4]{AsKO}).

$\bullet$ Define $O^{p'}(\mathcal{F})$ to be the minimal normal subsystem of $\mathcal{F}$ which has index prime to $p$ in $\mathcal{F}$.
\end{definition}

Now, we introduce  constrained fusion systems.  For the theory of
constrained fusion systems, we refer to \cite{AsKO, BCGLO, BLO2}. And the definition of component of fusion system is
due to \cite{As5, As6}.

\begin{definition}\cite{AsKO, BCGLO} A saturated fusion system
$\mathcal{F}$ is $constrained$ if $\mathcal{F}$ contains a normal centric $p$-subgroup, i.e.,
$O_{p}(\mathcal{F})$ is centric.
\end{definition}

\begin{theorem}(Model theorem for constrained fusion systems \cite[III, 5.10]{AsKO},\cite{BCGLO}. Let $\mathcal{F}$ be a
constrained, saturated fusion system over a $p$-group $S$. Fix $Q\in \mathcal{F}^c$ such that $Q\unlhd \mathcal{F}$. Then
the following hold.

(a) There is a model for $\mathcal{F}$: a finite group $G$ with $S\in \mathrm{Syl}_p(G)$ such that $Q\unlhd G$, $C_G(Q) \leq Q,$
and $\mathcal{F}_S(G) = \mathcal{F}$.

(b) For any finite group $G$ such that $S\in \mathrm{Syl}_p(G)$ such that $Q\unlhd G$, $C_G(Q) \leq Q,$ and
$\mathrm{Aut}_{G}(Q) = \mathrm{Aut}_{\mathcal{F}}(Q)$, there is $\beta \in \mathrm{Aut}(S)$ such that
$\beta|_Q = \mathrm{Id}_Q$ and $\mathcal{F}_S(G) =~ ^{\beta}\mathcal{F}$.

(c) The model $G$ is unique in the following strong sense: if $G_1, G_2$ are two finite groups such
that $S\in \mathrm{Syl}_p(G_i)$, $Q\unlhd G_i$, $\mathcal{F}_S(G_i) = \mathcal{F}$,
and $C_{G_i}(Q) \leq Q,$ for $i = 1, 2$, then there is an
isomorphism $\psi: G_1\longrightarrow G_2$ such that  $\psi|_S = \mathrm{Id}_S.$ If $\psi$  and  $\psi'$ are two such isomorphisms,
then  $\psi' = \psi \circ c_z$ for some $z\in Z(S)$.
\end{theorem}

\begin{theorem}\cite[Theorem 1]{As5}  Let $\mathcal{F}$ be a
constrained, saturated fusion system over a finite $p$-group $S$, $G$ a model of $\mathcal{F}$
and $\mathcal{E}\unlhd \mathcal{F}$. Then there is a unique normal subgroup of $G$ which is
a model of $\mathcal{E}$.
\end{theorem}

\begin{thmb*} Let $G$ be a finite solvable group and $G$ has isomorphic subgroup $H$ and $K$.
Let $H$ is maximal subgroup of $G$, we can set $|G: H|=p^n$. Let $p\leq 3$ and $q=5-p$. Let
$Q\in Syl_q(H)$. If $\mathcal{F}_Q(H)\unlhd \mathcal{F}_Q(G)$, then $K$ is also maximal.
\end{thmb*}

\begin{proof} Suppose that $(G, H, K)$ is a counterexample. Since $H$ is maximal in a solvable group
$G$, we can set $|G:H|=p^n$ for some prime $p$ and positive integer $n$.

{\bf Case 1.} $O_p(G)\neq 1$. By \cite[Theorem 3]{IR}, we have $O_{p}(G)\leq H$.
By \cite[Lemma 2]{IR}, we have
$$O_{p}(G)=O_{p}(G)\cap H= O_{p}(H),~~~~O_{p}(G)\cap K= O_{p}(K).$$
Since $H\cong K$, we have $O_{p}(H)\cong O_{p}(K)$. Hence, $O_{p}(G)\leq K$.
Now, we focus on $(G/O_{p}(G), H/O_{p}(G), K/O_{p}(G))$, we can see that
$K/O_{p}(G)$ is maximal in $G/O_{p}(G)$ because $(G, H, K)$ is a counterexample. So $K$ is maximal in $G$.
That is a contradiction.

{\bf Case 2.} $O_p(G)=1$. First, since $G$ is solvable, we have
 $O_{p'}(G)\neq 1$. And $O_{p'}(G)\leq H$ because $|G: H|=p^n$.

By \cite[Theorem A]{GR}, we can see that $O_{q'}(G)=1$ because $(G, H, K)$ is a counterexample.
So $F(G)=O_{q}(G)$ and $O_{q}(G)\neq 1$ because $O_{p'}(G)\neq 1$. Since $C_{G}(O_{q}(G))\leq O_{q}(G)$,
it implies $G$ is a model of fusion system
$\mathcal{F}_Q(G)$. Since $\mathcal{F}_Q(H)\unlhd \mathcal{F}_Q(G)$, thus there exists a normal subgroup
$U$ of $G$ such that $$\mathcal{F}_Q(H)=\mathcal{F}_Q(U)$$
by \cite[Theorem 1]{As5}.

Since $\mathcal{F}_Q(H)=\mathcal{F}_Q(U),$ we have
$$\mathrm{Aut}_H(O_q(G))=\mathrm{Aut}_U(O_q(G)).$$
So for each $h\in H$, we have $c_h|_{O_q(G)}=c_u|_{O_q(G)}$ for some $u\in U$. That means
$$hu^{-1}\in C_G(O_{q}(G))\leq O_q(G)\leq H\cap U.$$
Hence, $H=U\unlhd G$.
Since $G/H$ is a $p$-group, we have that $|G/H|=p$ because $H$ is maximal in $G$. Hence, $K$ is maximal in $G$.
That is a contradiction.

So, we complete the proof.
\end{proof}

\begin{theorem} Let $G$ be a finite solvable group and $G$ has isomorphic subgroup $H$ and $K$.
Let $H$ is maximal subgroup of $G$, we can set $|G: H|=p^n$. Let $p\leq 3$ and $q=5-p$. Let
$Q\in Syl_q(H)$. Set $\mathcal{F}:=\mathcal{F}_Q(G)$. If $ O^{q'}(\mathcal{F})\geq\mathcal{F}_Q(H)$ and $ O^{q}(\mathcal{F})=\mathcal{F}$, then $K$ is also maximal.
\end{theorem}

\begin{proof} Suppose that $(G, H, K)$ is a counterexample. Since $H$ is maximal in a solvable group
$G$, we can set $|G:H|=p^n$ for some prime $p$ and positive integer $n$.

{\bf Case 1.} $O_p(G)\neq 1$. By \cite[Theorem 3]{IR}, we have $O_{p}(G)\leq H$.
By \cite[Lemma 2]{IR}, we have
$$O_{p}(G)=O_{p}(G)\cap H= O_{p}(H),~~~~O_{p}(G)\cap K= O_{p}(K).$$
Since $H\cong K$, we have $O_{p}(H)\cong O_{p}(K)$. Hence, $O_{p}(G)\leq K$.
Now, we focus on $(G/O_{p}(G), H/O_{p}(G), K/O_{p}(G))$, we can see that
$K/O_{p}(G)$ is maximal in $G/O_{p}(G)$ because $(G, H, K)$ is a counterexample. So $K$ is maximal in $G$.
That is a contradiction.

{\bf Case 2.} $O_p(G)=1$. Since $G$ is solvable, we have
$F(G)=O_{q}(G)=O_{p'}(G)\neq 1$. And $C_{G}(O_q(G))\leq O_q(G)$. So $G$ is a model of fusion system
$\mathcal{F}_Q(G)$. Since $O^{q'}(\mathcal{F})\unlhd \mathcal{F}_Q(G)$ and $ O^{q}(\mathcal{F})\unlhd \mathcal{F}_Q(G)$, thus there exist normal subgroup
$U$ of $G$ such that $$O^{q'}(\mathcal{F})=\mathcal{F}_Q(U)$$
by \cite[Theorem 1]{As5}.

We have $O_{p'}(G)\leq H$ because $|G: H|=p^n$.
Similarly, we have $O_{p'}(G)\leq K$ because $|G:K|=|G:H|=p^n$.

Since $\mathcal{F}_Q(U)= O^{q'}(\mathcal{F})\geq\mathcal{F}_Q(H),$ we have
$$\mathrm{Aut}_U(O_q(G))\geq \mathrm{Aut}_H(O_q(G)).$$
So for each $h\in H$, we have $c_u|_{O_q(G)}=c_h|_{O_q(G)}$ for some $u\in U$. That means
$$hu^{-1}\in  C_G(O_{q}(G))\leq O_q(G)\leq Q\leq U.$$
Hence, $H\leq U.$ Since $H$ is maximal in $G$, we have $U=H$ or $U=G$.
If $H=U\unlhd G$, we have $K$ is also maximal in $G$ by above theorem.
That is a contradiction. So, we have $U=G.$ That means $\mathcal{F}=O^{q'}(\mathcal{F})$.

Since $O^{q}(\mathcal{F})=\mathcal{F}$, we have $\mathcal{F}$ is not Puig-solvable. But $G$ is a model of $\mathcal{F}$ and $G$ is solvable, we
can see that $\mathcal{F}$ is Puig-solvable by \cite[Part II, Theorem 12.4]{AsKO}.
That is a contradiction.

So, we complete the proof.
\end{proof}

\textbf{ACKNOWLEDGMENTS}\hfil\break
The authors would like to thank  Prof. C. Broto for his discussion on the definition of characteristic p type group.
The authors would like to thank Southern University of Science and Technology for their kind hospitality hosting joint meetings of them.

\end{document}